\documentclass[a4paper]{article}
\usepackage[utf8]{inputenc}
\usepackage{amssymb,amsthm,amsmath,graphicx,
bm,ebezier,mathtools,amsfonts}
\usepackage{color}
\usepackage[noend]{algpseudocode}
\usepackage[ruled,vlined,linesnumbered]{algorithm2e}
\usepackage{hyperref}
\usepackage{enumitem}
\usepackage{amsmath}
\usepackage{breqn}
\usepackage{todonotes}

\newtheorem{theorem}{Theorem}[section]

\newtheorem{proposition}[theorem]{Proposition}
\newtheorem{corollary}[theorem]{Corollary}
\newtheorem{lemma}[theorem]{Lemma}

\theoremstyle{remark}
\newtheorem{example}[theorem]{Example}
\newtheorem{remark}[theorem]{Remark}

\newcommand{\N}{\mathbb{N}}

\newcommand{\Q}{\mathbb{Q}}
\newcommand{\Z}{\mathbb{Z}}
\def\k{\mathbb{K}}
\renewcommand{\mod}{ \mathrm{mod}\,}
\newcommand{\tZ}{\mathrm{Z}}
\newcommand{\CaC}{\mathcal{C}}
\newcommand{\CaD}{\mathcal{D}}

\title{On $p$-Frobenius of affine semigroups}
\author{Evelia R. Garc\'{\i}a Barroso, J. I. Garc\'{\i}a-Garc\'{\i}a, \\Luis J. Santana S\'anchez and Alberto Vigneron-Tenorio}
\date{}

\definecolor{darkgreen}{RGB}{6, 66, 22}

\begin{document}

\maketitle
\begin{abstract}
The aim of this paper is to study the $p$-Frobenius vector of affine semigroups $S\subset \mathbb N^q$; that is, the maximum element, with respect to a graded monomial order, with at most $p$ factorizations in $S$. We produce several algorithms to compute these vectors. Finally, we study how the $p$-Frobenius vectors behave when considering gluings of $S$ with $\mathbb N^q$.
\end{abstract}

\section*{Introduction}

An affine semigroup $S\subset \N^q$ is a set containing $0$ and closed under addition. A finite set $A=\{a_1,\ldots ,a_h\}\subset \N^q$ is a generating set of $S$ if $S=\big\{\sum_{i=1}^h \lambda_i a_i\mid \lambda_1,\ldots ,\lambda_h\in \N \big\}$. It is called a minimal generating set if it is the minimal set, according to inclusion, generating $S$. In this work, $S=\langle A\rangle$ means that $A$ is the minimal set of generators of $S$. In what follows, when we talk about an affine semigroup, we must understand a finitely generated affine semigroup.

Let $S=\langle A\rangle$ and $n\in \N^q$, the set $\tZ_n(S)$ denotes $\{\lambda=(\lambda_1,\ldots ,\lambda_h)\in \N^h\mid  n=\sum_{i=1}^h \lambda_i a_i\}$. The minimum integer cone containing $S$ is $\CaC(S)=\big\{\sum_{i=1}^h \lambda_i a_i\mid \lambda_1,\ldots ,\lambda_h\in \Q_{\ge 0}\big\}\cap \N^q$.
We say that $S$ is a $\CaC$-semigroup if $\CaC(S)\backslash S$ is a finite set. For $q=1$, $S$ is called a numerical semigroup when $\N\setminus S$ is finite (equivalently, $\gcd(a_1,\ldots ,a_h)=1$).

An important invariant related to numerical semigroups is the Frobenius number, defined as the maximum element in $\Z\setminus S$, that is, the largest integer that cannot be written as a positive linear combination of the minimal generators of $S$. Observe that $f$ is the Frobenius number of $S$ if and only if $f$ is the maximum integer satisfying $\tZ_f(S)=\emptyset$. Thus, one may naturally extend this definition to affine semigroups and call the Frobenius vector, the {\it maximum} (for a fixed monomial order $\preceq$) integer vector satisfying $\tZ_f(S)=\emptyset$. However, this maximum element might not exist for several reasons. The worst case arises when $\{f\in \CaC(S)\mid \tZ_f(S)=\emptyset\}$ is not finite. Nevertheless, when $\CaC(S)\setminus S$ is finite, this {\it maximum} integer vector can be set by $\max_{\preceq} (\CaC(S)\setminus S)$ for the fixed  monomial order $\preceq$ (\cite{Csemigroups}). In \cite{SL22}, the possible Frobenius vectors for an affine semigroup $S$ such that $\CaC(S)=\N^q$ and $\CaC(S)\setminus S$ finite are studied.

The first generalization of the Frobenius number appeared in \cite{BeckRobins}, but later many other generalizations of the Frobenius number/vector have been introduced. For a numerical semigroup $S$, the most usual definitions of generalized Frobenius number (called $p$-Frobenius number) are: the largest integer $n\in \N$ such that $\#
\tZ_n(S)=p$ (see \cite{BrownDannenbergFoxHannaKeckMooreRobbinsSamplesStankewicz}), or the largest integer $n\in \N$ such that $\# \tZ_n(S)\le p$ (see \cite{KomatsuYing} and references therein). These definitions are also used for affine semigroups. In particular,  a $p$-Frobenius integer number associated with an affine semigroup was introduced in \cite{AlievDeLoeraLouveaux}.

In this work, for an affine semigroup $S$, we introduce the concept of $p$-Frobenius vector of $S$ (respect to a {\it graded} monomial order $\preceq$), which is defined as $F_0(S)=\max_\preceq\{\CaC(S)\setminus S\}$, and $F_p(S)=\max_\preceq\{n\in \CaC(S)\mid 0< \sharp \tZ_n(S)\le p\}$, for $p>0$. When the set defining $F_0(S)$ is not finite, we set $F_0(S)=(\infty,\ldots, \infty)$. Similarly, for $p>0$, $F_p(S)=(\infty,\ldots, \infty)$ when its defining set is not finite. One of the goals of our work is to characterize when $F_p(S)$ is finite (Theorem \ref{existenciaFp}), and, as a consequence, to provide an algorithm to compute the $p$-Frobenius vector from the minimal generating set of any affine semigroup (Algorithm \ref{computeFp}). Moreover, we give two improved algorithms for the cases $p=1$ and $p=2$. The case $p=0$ was solved in \cite{DiazGarciaMarinVigneron}. 
In that paper, the authors characterize the affine semigroups $S$ such that $\CaC(S)\setminus S$ is finite, and an algorithm to compute its gap sets is introduced. For both results, only a generating set of $S$ is required.

The other target of this paper is related to the gluing of semigroups. The {\it gluing} of numerical semigroups is a concept born from the study of complete intersection numerical semigroups (see \cite{Delorme}, \cite[Chapter 8]{R-GS} and references therein). This concept can be generalized to general affine semigroups (see \cite{AssiGarciaOjeda} and \cite{GarciaMorenoVigneron}). We consider the gluing of an affine semigroup with $\N^q$: given the affine semigroup $S =\langle a_1,\ldots,a_{h}\rangle\subset\mathbb N^q$, $d\in \N$ and $\gamma\in S\setminus \{a_1,\ldots,a_{h}\}$ with $d$ and $\gcd(\gamma_1,\ldots,\gamma_q)$ coprime, $S\oplus_{d,\gamma} \mathbb N^q$ is the affine semigroup minimally generated by $\{ da_1,\ldots, da_{h},\gamma\}$. We say that the semigroup $S\oplus_{d,\gamma} \mathbb N^q$ is an $\N^q$-gluing (affine) semigroup. An interesting paper on when two affine semigroups can be glued is \cite{GS2020}. In this context, it is well-known that, the Frobenius number of a numerical semigroup generated by two coprime elements $\{a,b\}$ is $(a-1)(b-1)$ \cite{Sylvester}), and it also exists an exact formula for the Frobenius number when the semigroup is generated by three elements (\cite{Tripathi}). Moreover, for numerical semigroups, the Frobenius number of $S\oplus_{d,\gamma} \mathbb N$ is determined by $dF_0(S)+(d-1)\gamma$ (\cite[Proposition 10]{Delorme}). In this work, we study some properties of $p$-Frobenius vector of $S'=S\oplus_{d,\gamma} \mathbb N^q$, and determine an explicit way to obtain it from the $p$-Frobenius vector of $S$ under certain conditions. If such conditions do not hold, an upper bound of $F_p(S')$ is provided. 

The content of this work is organized as follows: in Section \ref{sec_prelimiraries} we introduce the necessary background for the correct understanding of the work. In Section \ref{CalculandoFp} we characterize when the $p$-Frobenius vector is finite for any affine semigroup $S$, and we establish an algorithm to compute it. Sections \ref{F1} and \ref{F2} are devoted to improve the previous algorithm for $p=1,2$, respectively. In the last section (Section \ref{gluing}) we study the $p$-Frobenius vector of the semigroup obtained from the gluing of an affine semigroup with $\N^q$. The results of this work are illustrated with several examples.

\section{Preliminaries}\label{sec_prelimiraries}

For any $n\in \N\backslash \{0\}$, $[n]$ denotes the set $\{1,\ldots n\}$.

Given the minimal generating set $\{ a_1,\ldots,a_{h}\}$ of an affine semigroup $S$, and a field $\k$, we can consider the $S$-graded polynomial ring $\k[x_1,\ldots,x_h]$ where the $S$-degree of a monomial $X^\alpha= x_1^{\alpha_1}\cdots x_h^{\alpha_h}$ is $\sum _{i=1}^h\alpha_i a_i$. In this polynomial ring, we define the $S$-homogeneous polynomial ideal $I_S \subset \k[x_1,\ldots,x_h] $ as 
\begin{equation}\label{idealS}
I_S=\left\langle x_1^{\alpha_1}\cdots x_h^{\alpha_h}- x_1^{\beta_1}\cdots x_h^{\beta_h} \mid \sum _{i=1}^h\alpha_i a_i = \sum _{i=1}^h\beta_i a_i\right\rangle.
\end{equation}
This ideal is usually called the semigroup ideal of $S$. It is well-known (see \cite{Herzog}) that (pure) binomials finitely generate this ideal, and there exist some minimal generating sets with respect to inclusion.

In this work, we use several concepts and tools related to computational algebra. The reader can find the necessary background in \cite{CoxLittleOShea}, here we collect the essential definitions and properties to improve his reading.

Let $\preceq$ be a monomial order on $\k [x_1,\ldots, x_h]$, that is, a multiplicative total order on the set of monomials satisfying that for any two monomials $X^\alpha, X^\beta$ with $X^\alpha\prec X^\beta$, then $X^\alpha X^\gamma \prec X^\beta X^\gamma$ for every monomial $X^\gamma$. Given an ideal $I\subseteq \k [x_1,\ldots, x_h]$, we denote by $\mathrm{In}_\prec(I)$  the set of leading terms of non-zero elements of $I$, and $\langle \mathrm{In}_\prec(I)\rangle$ the monomial ideal generated by $\mathrm{In}_\prec(I)$. A finite subset $G$ of $I$ is a Gr\"{o}bner basis of $I$ if $\langle \mathrm{In}_\prec(I)\rangle = \langle \{ \mathrm{In}_\prec(g) \mid g\in G\}\rangle$, where $\mathrm{In}_\prec(g)$ is the leading term of $g$. A Gröbner basis is reduced if all its polynomials are monic and irreducible by its other polynomials. This reduced basis is unique for each order. An algorithm for computing the (reduced) Gr\"{o}bner bases for $I$ is given in \cite[Chapter~2, $\S$7]{CoxLittleOShea}. It is also well-known that Gr\"{o}bner bases of monomial (resp. binomial) ideals are sets of monomials (resp. binomials).

Given a monomial order $\preceq$ and a Gröbner basis $G$, we denote by $\texttt{NormalForm}_{\preceq}(f,G)$ the remainder of the division of $f\in \k[x_1,\ldots ,x_h]$ by $G$ with respect  to $\preceq$. Since $G$ is a Gröbner basis, $\texttt{NormalForm}_{\preceq}(f,G)$ is unique (see \cite[Chapter 2, $\S$6, Proposition 1]{CoxLittleOShea}).
Taking into account the definition of $F_p(S)$ given in the introduction, we not only assume that the order is a monomial order but a graded one as well. That is, the monomials are first compared by total degree, with ties broken by one other order (see \cite[Chapter 8, $\S$4, Proposition 1]{CoxLittleOShea}). This is to avoid the following contradictory cases: for instance, consider the affine semigroup $S$ generated by $\{(0,1),(1,1),(2,0),(3,0)\}$ and the lexicographical order $\preceq_{lex}$; we have that $\alpha(0,1)$ has a unique writing for any $\alpha \in \N$, meaning that $\{n\in \CaC(S)\mid \sharp \tZ_n(S) =1\}$ is an infinite set so $F_1(S)=(\infty, \infty)$ according to definition; nevertheless, it has a maximum with respect to $\preceq_{lex}$, which is $(7,0)$. Indeed, {\it graded} monomial orders satisfy that $\sharp \{s\in S\mid s\preceq a\}<\infty$ for every $a\in S$.

\section{Computing $F_p(S)$ }\label{CalculandoFp}

Consider $S\subset \N^q$ an affine semigroup and $\preceq$ a graded monomial ordering on $\N^q$. 

In this section, we provide an algorithm to compute $F_p(S)$  for any $p\in \N$. Recall that to solve the problem for $p=0$, you can use the results appearing in \cite{DiazGarciaMarinVigneron}. The following result is the key to obtain such an algorithm for $p\ge 1$. Furthermore, it characterizes when $F_p(S)\in \N^q$.

\begin{theorem}\label{existenciaFp}
Let $S=\langle a_1,\dots,a_h\rangle\subset \N^q $ be an affine semigroup, $p\in\N\setminus\{0\}$, and $\preceq$ a graded monomial ordering on $\N^q$.
Then, $F_p(S)\neq (\infty,\ldots ,\infty)$ if and only if, for every $k\in [h]$, there exist $\lambda_k, \alpha_{k,i} \in \N$ such that $\lambda_k a_k=\sum_{i=1,i\neq k}^{h}\alpha_{k,i}a_i$.

\end{theorem}
\begin{proof}
Without loss of generalization, assume that for $k=1$, for every $\lambda\in\N$ the element $\lambda a_1$ cannot be expressed using only the generators $\{a_2,\dots,a_h\}$. This implies that $Z_{\lambda a_1}(S)=\{(\lambda,0,\dots,0)\}$ (the only expression of $\lambda a_1$ is itself). Therefore, $\# Z_{\lambda a_1}(S)=1$ for every $\lambda \in \N$, and thus $F_p(S)=(\infty,\ldots,\infty)$.

Conversely, let $b=\sum_{i=1}^{h}\mu_i a_i\in S$ such that there exists $k\in [h]$ with $\mu_k\geq p \lambda_k$, assume $k=1$. We have that $\mu_1=p\lambda_1+d$ with $d\in \N$, and 
    \[\begin{split}
    b=(p\lambda_1+d)a_1+\sum_{i=2}^{h}\mu_i a_i=((p-1)\lambda_1+d)a_1+\sum_{i=2}^{h}\mu_i a_i+\sum_{i=2}^{h}\alpha_{1,i}a_i=\\
    ((p-2)\lambda_1+d)a_1+\sum_{i=2}^{h}\mu_ia_i+\sum_{i=2}^{h}2\alpha_{1,i}a_i=\\
    \dots\\
    da_1+\sum_{i=2}^{h}\mu_ia_i+\sum_{i=2}^{h}p\alpha_{1,i}a_i.
    \end{split}\]
Obtaining in this way, $p+1$ different factorizations of $b$. Thus, all the elements with at most $p$ factorizations are in the bounded set $\{\sum_{i=1}^h \gamma_i a_i \mid \gamma_i\in \N,~\gamma_i\leq p\lambda_i \}$, and therefore $F_p(S)\in \N^q$.
\end{proof}

\begin{remark} As a consequence of Theorem \ref{existenciaFp}, we see that the finiteness of $F_p(S)$ is independent of $p$ and the graded monomial order considered. In particular, if $F_1(S)$ is finite, then $F_p(S)$ is as well for any $p\geq 1$.

\end{remark}

\begin{corollary}
Let $S=\langle a_1,\dots,a_h\rangle\subset \N^q $ be an affine semigroup, $p\in\N\setminus\{0\}$, and $\preceq$ a graded monomial ordering on $\N^q$.
Then, $F_p(S)\neq (\infty,\ldots ,\infty)$ if and only if every extremal ray $\tau$ of $\CaC(S)$ contains at least two minimal generators of $S$.

\end{corollary}
\begin{proof}
If there is an extremal ray that contains only one minimal generator, then that generator cannot be written as a combination of the other minimal generators. Thus, by Theorem \ref{existenciaFp}, $F_p(S)$ is not finite.

Assume now that, every extremal ray $\tau$ of $\CaC(S)$ contains at least two minimal generators of $S$. Let $k \in [h]$, by Theorem \ref{existenciaFp}, it is enough to prove that a multiple of $a_k$ is a combination of the other minimal generators to conclude. If $a_k$ is not in any extremal ray of $\CaC(S)$ then there exists $\lambda_k\in\N$ such that $\lambda_k a_k$ can be expressed using only the generators of $S$ belonging to the extremal rays of $\CaC(S)$. Otherwise, if $a_k$ is in an extremal ray $\tau$ of $\CaC(S)$, by hypothesis, there exists another minimal generator $a_j \in \tau$, thus $a_k = \alpha a_j$ for some $\alpha \in \Q$. Hence, there exist $\beta, \gamma \in \N$ such that $\beta a_k = \gamma a_j$. 
\end{proof}

Given a $h$-tuple $\Lambda=(\lambda_1,\ldots,\lambda_h) \in \N^h$, once obtained the set,
\begin{equation}\label{Dset}
\CaD(\Lambda,p):=\left\{\sum_{i=1}^h \gamma_i a_i \mid \gamma_i\in \N,~\gamma_i\leq p\lambda_i \right\},
\end{equation}
the algorithm to compute $F_p(S)$ is straightforward from the proof of Theorem \ref{existenciaFp}. 

Algorithm \ref{computeFp} computes $F_p(S)$. Note that, if Theorem \ref{existenciaFp} holds, then any Gröbner basis of the ideal $I_S\subset \k[x_1,\ldots ,x_h]$ associated with $S$ contains a binomial with a monomial like $x_k^{\alpha_k}$, for all $k\in[h]$. Moreover, in this procedure, the elements $\lambda_k$ obtained are the smallest elements satisfying $\lambda_k a_k=\sum_{i=1,i\neq k}^{h}\alpha_ia_i$.

\begin{algorithm}[H]
\caption{Computation of $F_p(S)$.}\label{computeFp}
\KwIn{A minimal system of generators $\{a_1,\dots,a_h\}$ of $S$ and $p\in \N$.}
\KwOut{$F_p(S)$.}
\If{$p=0$}{
    \If{$S$ is a numerical semigroup}{ 
        \Return{The Frobenius number of $S$}
    }
    \If{$\CaC(S)\setminus S$ is finite}{
        \Return{The Frobenius vector of $S$}
    }
    \If{$\CaC(S)\setminus S$ is not finite}{
        \Return{$(\infty,\dots,\infty)$}
    }
}
\If{there is an extremal ray of $\CaC(S)$ with only one minimal generator of $S$}
    {\Return{$F_p(S)=(\infty,\dots,\infty)$} }
${\cal B}\leftarrow $ a (reduced) Gröbner basis of $I_S$\;

$\Lambda\leftarrow (\lambda_1,\ldots , \lambda_h)\in \N^h$ such that $x_k^{\lambda_k}$ is a monomial of a binomial in ${\cal B}$\;
$D\leftarrow \CaD(\Lambda,p)$\;

\Return{$F_p(S)=\max_\preceq\{n\in D\mid 0<\# \tZ_n(S)\le p\}$}
\end{algorithm}

\begin{example}\label{ex1}
    Let $S$ be the affine semigroup generated by the elements of the set $A=\{(3, 0), (4, 0), (0, 5), (0, 6), (1, 1)\}$ and consider the graded lexicographic order the monomial ordering used in this example. The set $A$ is a minimal generating set of $S$ and the ideal of the semigroup is generated by
    \begin{multline*}
        \Big\{x_3^6-x_4^5,x_2 x_4^4-x_3^4 x_5^4,x_2 x_3^2-x_4 x_5^4,
        x_2^2 x_4^3-x_3^2 x_5^8,\\
        x_2^3 x_4^2-x_5^{12},
        x_1 x_5^5-x_2^2 x_3,
        x_1 x_4^3-x_3^3 x_5^3,x_1 x_3 x_5-x_2 x_4,x_1 x_3^3-x_4^2 x_5^3,\\
        x_1 x_2 x_4^2-x_3 x_5^7,
        x_1 x_2^2 x_3 x_4-x_5^{11},x_1^2 x_4-x_5^6,x_1^3 x_3-x_2 x_5^5,x_1^4-x_2^3\Big\}.
    \end{multline*}
    Note that the monomials $x_1^4$, $x_2^3$, $x_3^6$, $x_4^5$ and $x_5^{11}$ appear in the above set. Thus $\Lambda=(4,3,6,5,11)$. The set $\CaD(\Lambda,1)$ is equal to 
    \[\{ (3\gamma_1 +4\gamma_2+\gamma_5,5\gamma_3+6\gamma_4+\gamma_5)\mid (0,0,0,0,0)\leq \gamma\leq (4,3,6,5,11)\},\]
    containing $1835$ elements. Ordering this set with respect to the fixed monomial order, the greatest element $n$ having $\#Z_n(S)$ equal to $1$ is $(21,4)$ and therefore $F_1(S)=(21,4)$. 
    
\end{example}

\section{Computing the 1-Frobenius vector from Gröbner basis of $I_S$}\label{F1}

Fix a graded monomial order $\preceq$ in $\N ^q$ and an affine semigroup $S\subset \N ^q$ minimally generated by $\{a_1,\dots,a_h\}$, Algorithm \ref{computeFp} allows us to obtain $F_p(S)$, but it involves computing the sets $\tZ_m(S)$ for $m\in \CaD(\Lambda,p)$, and $\CaD(\Lambda,p)$ could be very large. In this section, we introduce an algorithm to improve the computation of $F_1(S)$ by using some computational algebra tools.

In general, given any monomial $X^\alpha\in \k[x_1,\ldots ,x_h]$ with $S$-degree $m$, and any Gröbner basis ${\cal B}$ of $I_S$ with respect to a monomial order $\preceq'$ $\texttt{NormalForm}_{\preceq'}(X^\alpha,{\cal B})\neq X^\alpha$ implies that $\# \tZ_m(S)> 1$. Indeed, since the elements of $\cal B$ are binomials, then $\texttt{NormalForm}_{\preceq'}(X^\alpha,{\cal B})$ is a monomial, say, $X^\beta$; thus, $X^\alpha - X^\beta \in I_S$, meaning that they have the same $S$-degree, so $\alpha, \beta \in \tZ_m(S)$. Hence, if you consider the following set (equivalent to $\CaD(\Lambda,1)$),
\[ \CaD'(\Lambda)=\left\{(\gamma_1,\ldots ,\gamma_h)\mid \gamma_i\in \N,~\gamma_i\leq \lambda_i \right\},\]
the elements $m_\gamma=\sum_{i=1}^h\gamma_ia_i\in S$ with $\gamma\in \CaD'(\Lambda)$ and such that $\# \tZ_{m_\gamma}(S)=1$ satisfy $\texttt{NormalForm}_{\preceq'}(X^\gamma,{\cal B})= X^\gamma$. This fact can be used to improve Algorithm \ref{computeFp} for $p=1$.

The following lemma also improves Algorithm \ref{computeFp} for $p=1$. Consider that ${\cal B}=\{X^{u_1}-X^{v_1},\ldots ,X^{u_t}-X^{v_t}\}$ is the reduced Gröbner basis of $I_S$ respect to $\preceq'$, and let $I_v\subset \k[x_1,\ldots ,x_h]$ be the monomial ideal generated by $\{X^{v_1},\ldots ,X^{v_t}\}$. We assume the leading term of $X^{u_i}-X^{v_i}$ is $X^{u_i}$ for any $i$.

\begin{lemma}
Let $\alpha=(\alpha_1, \ldots, \alpha_h) \in \CaD'(\Lambda)$ such that $\texttt{NormalForm}_{\preceq'}(X^\alpha,{\cal B})=X^\alpha$. Then, $\# \tZ_{\sum_{i=1}^h\alpha_i a_i}(S)>1$ if and only if $X^\alpha\in I_v$.
\end{lemma}

\begin{proof}
Since $\texttt{NormalForm}_{\preceq'}(X^\alpha,{\cal B})=X^\alpha$, $X^\alpha$ is not the leading term of any binomial in $I_S$ respect to $\preceq'$.

Assume $\# \tZ_{\sum_{i=1}^h\alpha_i a_i}(S)>1$. Then, there is $\beta \in \N^h$ with $\sum_{i=1}^h\beta_i a_i=\sum_{i=1}^h\alpha_i a_i$, and $\beta\neq \alpha$, that is, $X^\beta-X^\alpha\in I_S$. Hence, there exist $f_1,\ldots ,f_t\in \k[x_1,\ldots ,x_h]$ such that $X^\beta - X^\alpha = \sum _{i=1}^t f_i (X^{u_i}-X^{v_i})$. So, $X^\beta$ is the leading term of $X^\beta-X^\alpha$, and then, $X^\alpha$ has to be equal to $X^\gamma X^{v_i}$ for some $i\in [t]$ and $X^\gamma\in \k[x_1,\ldots ,x_h]$. We have that $X^\alpha\in I_v$.

Conversely, suppose $X^\alpha\in I_v$, so $X^\alpha$ is equal to $X^\beta X^{v_j}$ for some $j\in [t]$ and $X^\beta\in \k[x_1,\ldots ,x_h]$. Hence, $\sum_{i=1}^h\alpha_i a_i= \sum_{i=1}^h(\beta_i+v_{ji}) a_i$, where $v_j=(v_{j1},\ldots ,v_{jh})$. On the other hand, $X^{u_j}-X^{v_j}\in I_S$ and so $\sum_{i=1}^h(\beta_i+v_{ji}) a_i=\sum_{i=1}^h(\beta_i+u_{ji}) a_i $ with $u_j=(u_{j1},\ldots ,u_{jh})$. Thus $\alpha=\beta+v_j$ and $\beta + u_j$ are two different elements in $ \tZ_{\sum_{i=1}^h\alpha_i a_i}(S)$.
\end{proof}

\begin{algorithm}[H]
\caption{Improved computation of $F_1(S)$.}\label{computeF1}
\KwIn{A minimal system of generators $\{a_1,\dots,a_h\}$ of $S$.}
\KwOut{$F_1(S)$ }

\If{there is an extremal ray of $\CaC(S)$ with only one minimal generator of $S$}
    {\Return{$F_1(S)=(\infty,\dots,\infty)$.}}

${\cal B}\leftarrow $ a (reduced) Gröbner basis of $I_S$\;

$\Lambda\leftarrow (\lambda_1,\ldots , \lambda_h)\in \N^h$ such that $x_k^{\lambda_k}$ is a monomial of a binomial in ${\cal B}$\;

$D\leftarrow \{\gamma=(\gamma_1,\ldots ,\gamma_h)\in \CaD'(\Lambda)\mid \texttt{NormalForm}_{\preceq_k}(X^\gamma,{\cal B})=X^\gamma\}$\;
$D\leftarrow \{\gamma=(\gamma_1,\ldots ,\gamma_h)\in D\mid X^\gamma\notin I_v\}$\;

\Return{$F_1(S)=\max_\preceq\{\sum_{i=1}^h\gamma_i a_i\mid (\gamma_1,\ldots ,\gamma_h)\in D\}$}
\end{algorithm}

Note that the previous results mean that the set of elements $m\in S$ with $\# \tZ_{m}(S)=1$ corresponds to the set of monomials in $X^\alpha\in \k[x_1,\ldots ,x_h]$ such that $X^\alpha\notin \langle \mathrm{In}_\prec(I)\rangle+I_v$. Furthermore, the monomial ideal $\langle \mathrm{In}_\prec(I)\rangle+I_v$ does not depend on the fixed monomial order. This fact allows us to introduce an alternative algorithm to compute $F_1(S)$.

\begin{algorithm}[H]
    \caption{Improved (v2) computation of $F_1(S)$.}\label{computeF1-v2}
    \KwIn{A minimal system of generators $\{a_1,\dots,a_h\}$ of $S$.}
    \KwOut{$F_1(S)$.}
    \If{there is an extremal ray of $\CaC(S)$ with only one minimal generator of $S$}
    {\Return{$F_1(S)=(\infty,\dots,\infty)$.}}
    ${\cal B}\leftarrow$ a (reduced) Gröbner basis of $I_S$\;
    $\Omega\leftarrow \{\gamma \in\N^h \mid X^\gamma \text{ is a monomial of a binomial of }{\cal B}\}$\;
    $D \leftarrow \N^h\setminus \cup _{\gamma \in \Omega}(\gamma+ \N^h)$\;
    \Return{$F_1(S)=\max_\preceq\{\sum_{i=1}^{h}\gamma_ia_i\mid (\gamma_1,\dots,\gamma_h)\in D\}$}
\end{algorithm}

\begin{example}\label{ex2}
    Consider the same semigroup of Example \ref{ex1}. From the monomials of the Gröbner basis computed, we obtain the set of tuples
    \begin{multline*}
        \Omega=\{(0, 0, 0, 5, 0), (0, 0, 6, 0, 0), (0, 1, 0, 4, 0), 
        (0, 0, 4, 0, 4), (0, 1, 2, 0, 0), (0, 0, 0, 1, 4),\\
        (0, 2, 0, 3, 0), (0, 0, 2, 0, 8), (0, 3, 0, 2, 0), 
        (0, 0, 0, 0, 12), (0, 2, 1, 0, 0), (1, 0, 0, 0, 5),\\
        (1, 0, 0, 3, 0), (0, 0, 3, 0, 3), (0, 1, 0, 1, 0), 
        (1, 0, 1, 0, 1), (1, 0, 3, 0, 0), (0, 0, 0, 2, 3),\\
        (1, 1, 0, 2, 0), (0, 0, 1, 0, 7), (1, 2, 1, 1, 0),
        (0, 0, 0, 0, 11), (2, 0, 0, 1, 0), (0, 0, 0, 0, 6),\\ (3, 0, 1, 0, 0), (0, 1, 0, 0, 5), (0, 3, 0, 0, 0), 
        (4, 0, 0, 0, 0)\}
    \end{multline*}
    The set $ \{ \sum_{i=1}^5\alpha_i a_i | \alpha\in \N^5\setminus \cup _{\gamma \in \Omega}(\gamma + \N^5) \} $ has cardinality $179$ and its maximum with respect to the monomial order is $(21, 4)$. Thus $F_1(S)=(21,4)$, the same we obtained with Algorithm \ref{computeFp}.    
\end{example}

\section{2-Frobenius vector and indispensable binomials}\label{F2}

As in the previous sections, fix a graded monomial order $\preceq$ in $\N ^q$ and an affine semigroup $S$ minimally generated by $\{a_1,\dots,a_h\}$. It is well-known that all the minimal generating sets of the semigroup ideal $I_S$ have the same cardinality.  Moreover, these sets are characterized using simplicial complexes (see \cite{ComplejosSimpliciales}, and the references therein). Let $C_m\subset \k[x_1,\ldots ,x_h]$ be the set of monomials of $S$-degree $m\in S$. In \cite{TesisShalom}, it is introduced the simplicial complex $\nabla_m = \big\{ F \subseteq C_m \mid \gcd(F) \neq 1\big\},$ where $\gcd(F)$ denotes the greatest common divisor of the monomials in $F$, and $m$ belongs to $S$. Since $x_1^{\alpha_1}\cdots x_h^{\alpha_h}\in C_m$ if and only if $m=\sum_{i=1}^h \alpha_i a_i$, the vertex set of $\nabla_m$ consists of all the monomials of $S$-degree $m\in S$, which is equivalent to the set of all the ways of writing $m$ as a linear combination of the generators of $S$.

We have the characterization of the minimal generating sets of $I_S$,
\begin{theorem}(\cite{TesisShalom})
Let $\Lambda=\{X^{u_1}-X^{v_1},\ldots , X^{u_t}-X^{v_t}\}\subset I_S$, and $M=\{S\text{-degree}(X^{u_i})\mid i\in [t]\}$. Then, $\Lambda$ is a minimal generator set of $I_S$ if and only if,
\begin{enumerate}
    \item The simplicial complex $\nabla_{m}$ is non-connected for any $m\in M$.
    \item For any $m\in M$: 
    \begin{enumerate}
        \item The cardinality of $B_m$  is equal to the number of connected components of $\nabla_m$ minus one,
        \item the monomials $X^{u}$ and $X^{v}$ of any binomial $X^{u}-X^{v}\in B_m$  belongs to different connected components of $\nabla_m$,
        \item and there is at least a monomial of every connected component of $\nabla_m$,
    \end{enumerate}
    where $B_m$ is the set of binomials of $S$-degree $m$ in $\Lambda$.
\end{enumerate}
\end{theorem}

From the study of the uniqueness of the minimal generating set of $I_S$, the definition of the indispensable binomial of $I_S$ arises. An indispensable binomial of $I_S$ is a binomial that appears  (up to a scalar multiple) in every generator set of $I_S$. Equivalently, $X^{\alpha}-X^{\beta}\in I_S$ is an indispensable binomial if and only if $\nabla_m=\big\{ \{X^{\alpha}\},\{X^{\beta}\}\big\}$, where $m$ is the $S$-degree of $X^\alpha$ and $X^\beta$ (see Corollary 7 in \cite{indispensables}).

\begin{lemma}
Let $S$ be an affine semigroup such that there exists $m\in S$ with $\sharp \tZ_m (S)=2$. Then, there is at least an indispensable binomial in $I_S$.
\end{lemma}

\begin{proof}
The equality $\sharp \tZ_m(S)=2$ implies that either $\nabla_m=\big\{ \{X^{\alpha}\},\{X^{\beta}\}\big\}$ or $\nabla_m=\big\{ \{X^{\alpha}\},\{X^{\beta}\},\{X^{\alpha},X^{\beta}\}\big\}$. For the first case, $X^\alpha-X^\beta$ is an indispensable binomial in $I_S$, and for the second one, $\gcd(X^{\alpha},X^{\beta})^{-1}\big( X^\alpha-X^\beta\big)$ is an indispensable binomial.
\end{proof}

\begin{corollary}
Given $S$ an affine semigroup satisfying the hypothesis of Theorem \ref{existenciaFp}. If there is no indispensable binomial in $I_S$, then $F_1(S)=F_2(S)$.
\end{corollary}

From $\CaD(\Lambda,2)$, as defined in \eqref{Dset}, we set 
$$\CaD''(\Lambda)=\left\{ \gamma=(\gamma_1,\ldots ,\gamma_h)\mid \gamma_i\in \N,~\gamma_i\leq 2\lambda_i \right\}.$$

\begin{corollary}
Let $\gamma\in \CaD''(\Lambda)$ satisfying $\# \tZ_{\sum_{i=1}^h\gamma_i a_i}(S)=2$. Then, there exist $\gamma'\in \CaD''(\Lambda)$ and $X^\delta\in \k[x_1,\ldots, x_h]$ such that  $X^\gamma - X^{\gamma'}= X^\delta\big(X^\alpha-X^\beta\big)$ with $X^\alpha-X^\beta$ an indispensable binomial in $I_S$. 
\end{corollary}

As in the case $p=1$, Algorithm \ref{computeFp} can be improved for $p=2$ using results from this section. We can determine whether there are elements in the semigroup with only two ways of writing by checking whether there are any indispensable binomials in $I_S$.  Furthermore, if there are some indispensable binomials, we can significantly reduce the set of elements in the semigroup that can have two ways of writing.

\begin{algorithm}[H]
\caption{Improved computation of $F_2(S)$.}\label{computeF2}
\KwIn{A minimal system of generators $\{a_1,\dots,a_h\}$ of $S$.}
\KwOut{$F_2(S)$.}

\If{there is an extremal ray of $\CaC(S)$ with only one minimal generator of $S$}
    {\Return{$(\infty,\dots,\infty)$}}

\If{There is no indispensable binomial in $I_S$}
    {\Return{$F_2(S)=F_1(S)$}}

${\cal B}\leftarrow $ a (reduced) Gröbner basis of $I_S$\;
$\Lambda\leftarrow (\lambda_1,\ldots , \lambda_h)\in \N^h$ such that $x_k^{\lambda_k}$ is a monomial of a binomial in ${\cal B}$\;
$D \leftarrow \CaD''(\Lambda)$\;
$I \leftarrow$ the set of indispensable binomials in $I_S$\;

$G \leftarrow \{(\gamma,\gamma')\in \N^{2h} \mid X^\gamma - X^{\gamma'}\in I\}$\;

$D\leftarrow D\setminus \{\gamma,\gamma'\in \N^{h} \mid (\gamma,\gamma')\in G\}$\;

$G \leftarrow \{\gamma\in \N^{h} \mid (\gamma,\gamma')\in G \text{ for some } \gamma' \in \mathbb N^h\}$\;

\While{there is $\gamma,\gamma'\in D$ with $X^\gamma - X^{\gamma'}= bX^\delta$, such that $b\in I$}
    {
    \If{$\# \tZ_{\sum_{i=1}^h\gamma_i a_i}(S)=2$}
        {
       $G \leftarrow G\cup \{\gamma\}$\;
       $D\leftarrow D\setminus\{\gamma,\gamma'\}$\;
       }
    }
{$f\leftarrow \max_\preceq\{\sum_{i=1}^{h}\gamma_ia_i\mid (\gamma_1,\dots,\gamma_h)\in G\}$}\;
\Return{$F_2(S)=\max_\preceq\{F_1(S),f\}$}
\end{algorithm}
\begin{example}

Continuing with the semigroup of the examples \ref{ex1} and \ref{ex2}.
For this semigroup, all the binomials of the Gr\"obner basis given in Example \ref{ex1} are indispensable.

We have that $\Lambda = (4,3,6,5,11)$, and therefore the set $\CaD''(\Lambda)$ is equal to $\left\{ \gamma \in \N^p\mid \gamma \leq (8,6,12,10,22)\right\}$. We use the bound $(8, 6, 12, 10, 22)$ to compute the set $G$ from the set $D$ that initially contains $126721$ elements. These elements are the factorizations of $10071$ different elements in the monoid $S$. We search the maximum with respect to the graded lexicographic order having exactly $2$ factorizations. For this sake, we sort all the $10071$ different elements obtained in the semigroup, and, starting from the biggest one, the element $(70,164)$, we stop after we find an element with exactly $2$ factorizations. The element found is $(2,83)$ and thus $F_2(S)=(2,83)$. 
\end{example}

\section{$p$-Frobenius of $\N^q$-gluing affine semigroups}\label{gluing}
From now on, consider $S=\langle a_1,\ldots, a_h\rangle$ an affine semigroup, $d\in \N$, $\gamma=(\gamma_1,\ldots,\gamma_q)\in \N^q$ with $d$ and $\gcd(\gamma_1,\ldots,\gamma_q)$ coprime, and such that $S'=S\oplus_{d,\gamma} \mathbb N^q$ is an affine semigroup. In this section, we study the behavior of $F_p(S')$ with respect to $F_p(S)$.

For numerical semigroups, that is $q=1$, the relationship between $F_0(S')$ and $F_0(S)$ is the well-known formula $F_0(S')=dF_0(S)+(d-1)\gamma$. Note that, for $p=0$ and $q>1$, the $0$-Frobenius vector is finite if and only if $S$ is a $\CaC$-semigroup. Also note that for any $S$, $\CaC(S')\setminus S'$ is not finite. Therefore, we assume that $p\ge 1$.

\begin{lemma} \label{gamma coef}
    Let $s'=ds+a\gamma \in S'$ with $s\in S$ and $0\leq a \leq d-2$. Then, $\#Z_{s'}(S')=\#Z_{s'+\gamma}(S')$.
\end{lemma}
\begin{proof}
    Remember that $S'=\langle da_1, \ldots, da_h, \gamma \rangle$ and consider the injective map
\begin{align*}
\nu \colon Z_{s'}(S') &\longrightarrow  Z_{s'+\gamma}(S') \\
(\lambda_1, \ldots , \lambda_h, \lambda_\gamma) &\longmapsto (\lambda_1, \ldots , \lambda_h, \lambda_\gamma +1).
\end{align*}
To conclude, we prove that $\nu$ is surjective.

Take $\lambda'=(\lambda'_1, \ldots , \lambda'_h, \lambda'_\gamma) \in Z_{s'+\gamma}(S')$, that is,
\[
\lambda'_1 da_1 + \cdots + \lambda'_h da_h + \lambda'_\gamma \gamma = s'+\gamma = ds + (a+1)\gamma. 
\]
This implies that $(\lambda'_\gamma \gamma -(a+1)\gamma) \equiv 0 \; (\mod \;d)$, and since $\gcd(\gamma_1,\ldots,\gamma_q)$ and $d$ are coprime, it must be $\lambda'_\gamma \equiv a+1 \; (\mod \; d)$. In turn, since $\lambda'_\gamma \in \mathbb N$ and $0\le a \le d-2$, we must have $\lambda'_\gamma \geq a+1 >0$. Thus, $\lambda = (\lambda'_1, \ldots , \lambda'_h, \lambda'_\gamma-1) \in  Z_{s'}(S') $, and clearly, $\nu(\lambda)=\lambda'$.
\end{proof}

This allows us to give an upper bound for the $p$-Frobenius vector of  $S'$ when $p\ge 1$.

\begin{proposition}
     \label{gp: bound} Let $S$ and $S'$ as defined before and let $p\ge1$. Then $F_p(S')\preceq dF_p(S)+(d-1)\gamma$.
\end{proposition}

\begin{proof}
We have that $F_p(S') \in S'$, thus $F_p(S')=ds+a\gamma$, for some $s\in S$ and $a\in \mathbb N$. Since $\gamma \in S$, we can further assume that $0\leq a \leq d-1$. 

We claim that it must be $F_p(S')=ds+(d-1)\gamma$. Indeed, if $F_p(S')=ds+a\gamma$ with $a<d-1$, then, by Lemma \ref{gamma coef}, $F_p(S') + \gamma$ is also an element of $\{n\in \N\mid \# \tZ_n(S)\le p\}$, contradicting the $\preceq$-maximality of $F_p(S')$.

Finally, we prove that $s\preceq F_p(S)$. By contradiction, if $s \succ F_p(S)$, then $\# \tZ_s(S) \ge p+1$, and $\# \tZ_{sd+(d-1)\gamma}(S) \ge p+1$.
\end{proof}

The following result shows a necessary and sufficient condition for equality to hold in Proposition \ref{gp: bound} whenever $F_p(S)$ is such that $\#\tZ_{F_p(S)}(S)=p$.
\begin{theorem} \label{main gluing}
Assume that $\#\tZ_{F_p(S)}(S)=p$. Then,
$F_p(S')=dF_p(S)+(d-1)\gamma$ if and only if for every $b\in \tZ_\gamma(S)$ there is no $c\in \tZ_{F_p(S)}(S)$ such that  $b\le_{\N^h} c$, where $\le_{\N^h}$ denotes the partial order given by $b_i \le c_i$ for every $i \in \{1, \ldots, h\}.$
\end{theorem}
\begin{proof}

Let $s'=dF_p(S)+(d-1)\gamma$.

Assume there exist $b\in \tZ_\gamma(S)$ and $c\in \tZ_{F_p(S)}(S)$ such that  $b\le_{\N^h} c$. We show $s'\neq F_p(S')$ by proving $\# \tZ_{s'}(S') \geq p+1$. Indeed, notice that every $\lambda \in \tZ_{F_p(S)}(S)$ gives an element $(\lambda,d-1) \in \tZ_{s'}(S')$. This implies $\# \tZ_{s'}(S') \geq \# \tZ_{F_p(S)}(S)=p$. Moreover, one can check that $(c-b,2d-1) \in \tZ_{s'}(S')$ gives another element in the set. 

Conversely, let $\tZ_{F_p(S)}(S)=\{\lambda^{(1)}, \ldots, \lambda^{(p)}\}$. As just noticed, we have that
$$L=\left\{\left(\lambda^{(1)},d-1\right), \ldots, \left( \lambda^{(p)},d-1\right)\right\} \subseteq \tZ_{s'}(S').$$
Assume by contradiction that $F_p(S')\neq s'$. Then, by Proposition \ref{gp: bound} it must be $F_p(S')\prec s'$ . Thus,  $\#\tZ_{s'}(S')\geq p+1$. Let $\mu=(\mu_1, \ldots, \mu_h, \mu_\gamma) \in \tZ_{s'}(S')\setminus L$. Since $\gcd(\gamma_1,\ldots,\gamma_q)$ and $d$ are coprime, and
$$dF_p(S) + (d-1)\gamma = s' = \mu_1 da_1 + \cdots +\mu_h da_h + \mu_\gamma \gamma,$$
then  $\mu_\gamma \equiv d-1 \;(\mod \;d)$. If $\mu_\gamma = d-1$, then $(\mu_1, \ldots, \mu_h)\in \tZ_{F_p(S)}(S)$, which cannot be since $\mu \notin L$. Thus, $\mu_\gamma = kd+(d-1)$ for some $k\geq 1$. Taking $b=(b_1, \ldots, b_h) \in \tZ_\gamma (S)$, we have $(\mu_1+kb_1, \ldots, \mu_h +kb_h, d-1)\in \tZ_{s'}(S')$. Hence, $(\mu_1+kb_1, \ldots, \mu_h +kb_h) \in \tZ_{F_p(S)}(S)$, that is, $(\mu_1+kb_1, \ldots, \mu_h +kb_h)=c$ for some $c\in \tZ_{F_p(S)}(S)$. This implies that $b\leq_{\mathbb N^h} c$, which contradicts the hypothesis.
\end{proof}

\noindent {\bf Funding}. 
The first-named author was supported by the grant PID2019-105896GB-I00 funded by MCIN/AEI/10.13039/501100011033. 

The second and fourth-named authors were supported partially by Junta de Andaluc\'{\i}a research groups FQM-343.

Consejería de Universidad, Investigación e Innovación de la Junta de Andalucía project ProyExcel\_00868 also partially supported all the authors.

Proyecto de investigación del Plan Propio – UCA 2022-2023 (PR2022-011) partially supported the first, second and fourth-named authors.

Proyecto de investigación del Plan Propio – UCA 2022-2023 (PR2022-004) partially supported the second and fourth-named authors.

\medskip
\noindent
{\small Evelia Rosa Garc\'{\i}a Barroso\\
Departamento de Matem\'aticas, Estad\'{\i}stica e I.O. \\
Secci\'on de Matem\'aticas, Universidad de La Laguna\\
Apartado de Correos 456\\
38200 La Laguna, Tenerife, Spain\\
e-mail: ergarcia@ull.es}

\medskip

\noindent {\small Juan Ignacio Garc\'{\i}a-Garc\'{\i}a \\
Departamento de Matem\'aticas/INDESS (Instituto Universitario para el Desarrollo Social Sostenible)\\
Universidad de C\'adiz\\
 E-11510 Puerto Real, C\'adiz, Spain\\
e-mail: ignacio.garcia@uca.es}

\medskip

\noindent
{\small Luis Jos\'e Santana S\'anchez\\
IMUVA-Mathematics Research Institute, 
Universidad de Valladolid \\
47011 Valladolid, Spain\\
e-mail: luisjose.santana@uva.es}

\medskip

\noindent {\small Alberto Vigneron-Tenorio\\
Departamento de Matem\'aticas/INDESS (Instituto Universitario para el Desarrollo Social Sostenible)\\
Universidad de C\'adiz\\
E-11406 Jerez de la Frontera, C\'adiz, Spain\\
e-mail: alberto.vigneron@uca.es}

\end{document}